\documentclass[11pt, oneside]{article}   	
\usepackage{geometry}                		
\usepackage[dvipdfmx]{graphicx}
\usepackage[usenames]{xcolor}		
\usepackage{amsmath,amsthm,amssymb}
\usepackage{amscd}
\usepackage[all]{xy}

\theoremstyle{definition}
\newtheorem{dfn}{Definition}[section]
\newtheorem{rem}[dfn]{Remark}

\theoremstyle{plain}
\newtheorem{thm}[dfn]{Theorem}
\newtheorem{prop}[dfn]{Proposition}
\newtheorem{lem}[dfn]{Lemma}

\newcommand{\K}{\mathbb{K}}
\newcommand{\Z}{\mathbb{Z}}

\newcommand{\der}{\mathrm{Der}}
\newcommand{\aut}{\mathrm{Aut}}
\newcommand{\tder}{\mathrm{tder}}

\newcommand{\taut}{\mathrm{TAut}}

\renewcommand{\div}{\mathrm{div}}

\newcommand{\KV}{\operatorname{KV}}

\newcommand{\pa}{\partial}

\newcommand{\gr}{\mathrm{gr}}
\newcommand{\ad}{\mathrm{ad}}

\newcommand{\rot}{\mathrm{rot}}

\newcommand{\gfg}{\mathfrak{g}}
\newcommand{\omegaz}{{\omega_0}}
\newcommand{\ssym}{\mathrm{sym}}
\newcommand{\sym}{\mathrm{Sym}}

\title{Goldman-Turaev formality implies Kashiwara-Vergne}
\author{Anton Alekseev\thanks{Department of Mathematics, University of Geneva, 2-4 rue du Lievre, 1211 Geneva, Switzerland \texttt{e-mail:anton.alekseev@unige.ch}}, Nariya Kawazumi\thanks{Department of Mathematical Sciences, University of Tokyo, 3-8-1 Komaba, Meguro-ku, Tokyo 153-8914, Japan \texttt{e-mail:kawazumi@ms.u-tokyo.ac.jp}},
Yusuke Kuno\thanks{Department of Mathematics, Tsuda University, 2-1-1 Tsuda-machi, Kodaira-shi, Tokyo 187-8577, Japan \texttt{e-mail:kunotti@tsuda.ac.jp}} and Florian Naef\thanks{Department of Mathematics, Massachusetts Institute of Technology, 182 Memorial Dr, Cambridge, MA 02142, USA \texttt{e-mail:naeffl@mit.edu}}}

\date{}							

\begin{document}

\maketitle

\begin{abstract}
Let $\Sigma$ be a compact connected oriented 2-dimensional manifold with non-empty boundary.
In our previous work \cite{highergenus}, we have shown that the solution of generalized (higher genus) Kashiwara-Vergne equations for an automorphism $F \in \aut(L)$ of a free Lie algebra implies an isomorphism between the Goldman-Turaev Lie bialgebra $\mathfrak{g}(\Sigma)$ and its associated graded ${\rm gr}\, \mathfrak{g}(\Sigma)$. In this paper, we prove the converse: if $F$ induces an isomorphism $\mathfrak{g}(\Sigma) \cong {\rm gr} \, \mathfrak{g}(\Sigma)$, then it satisfies the Kashiwara-Vergne equations up to conjugation.
As an application of our results, we compute the degree one non-commutative Poisson cohomology of the Kirillov-Kostant-Souriau  double bracket.
The main technical tool used in the paper is a novel characterization of conjugacy classes in the free Lie algebra in terms of cyclic words.
\end{abstract}


\section{Introduction}
\label{sec:intro}

Let $\Sigma$ be a compact connected oriented surface with non-empty boundary and $\K$ a field of characteristic zero. The $\K$-linear space $\mathfrak{g}=\mathfrak{g}(\Sigma)$ spanned by free homotopy classes of loops in $\Sigma$ has an interesting Lie bialgebra structure, the Lie bracket being the Goldman bracket \cite{Go86} and the Lie cobracket being the Turaev cobracket \cite{Tu91}. (To be more precise, one needs to fix a framing on $\Sigma$ in order to define the Lie cobracket on $\mathfrak{g}$ and it actually depends on the choice of framing.) As was shown in \cite{highergenus}, one can naturally define the graded version $\gr\, \mathfrak{g}$ of the Goldman-Turaev Lie bialgebra, and it turns out to be isomorphic to the necklace Lie bialgebra structure \cite{BB02,Ginz01,Schedler} associated to a certain quiver determined by the topological type of $\Sigma$.

The Lie bialgebras $\mathfrak{g}$ and ${\rm gr}\, \mathfrak{g}$ admit natural completions which we denote by $\widehat{\mathfrak{g}}$ and ${\rm gr}\, \widehat{\mathfrak{g}}$, respectively.
They are isomorphic as filtered $\K$-vector spaces, but not canonically.
The formality question for the Goldman-Turaev Lie bialgebras is whether there exists a filtered Lie bialgebra isomorphism from $\widehat{\mathfrak{g}}$ to ${\rm \gr}\, \widehat{\mathfrak{g}}$ such that the associated graded map is the identity on ${\rm \gr}\, \widehat{\mathfrak{g}}$.
This question has been studied during the last several years by various approaches: the study first began with formality for Goldman brackets \cite{KK14,KK15,MT13,MTpre,Naef,Hain17}  and then has been deepened to formality for Turaev cobrackets \cite{Mas15,genus0,AN17,highergenus,Hain18}.
One motivation for considering this question comes from the study of the Johnson homomorphisms of mapping class groups \cite{KK16}.

In this paper, we impose a restriction on a map $\widehat{\mathfrak{g}} \to {\rm gr}\, \widehat{\mathfrak{g}}$ by assuming that it is induced by a group-like expansion \cite{Mas12}, which is a notion related to 1-formality of a free group of finite rank.
In order to explain this notion, for the moment we assume that the boundary of $\Sigma$ is connected.
(The general case needs a more careful treatment and will be explained in Section \ref{sec:results}.)
The group algebra $\K\pi$ of the fundamental group $\pi=\pi_1(\Sigma)$ has a decreasing filtration defined by powers of the augmentation ideal.
This defines a completion $\widehat{\K \pi}$, and the associated graded ${\rm gr}\, \widehat{\K \pi}$.
Since the fundamental group $\pi$ is a free group of finite rank, ${\rm gr}\, \widehat{\K \pi}$ is canonically isomorphic to the completed tensor algebra $\widehat{T}(H)$ generated by the first homology $H=H_1(\pi,\K)\cong H_1(\Sigma,\K)$.
Furthermore, we can identify $\widehat{\K \pi}$ with ${\rm gr}\, \widehat{\K \pi}$ in a non-canonical way.
In this context, a group-like expansion is a complete Hopf algebra isomorphism
\[
\theta\colon \widehat{\K \pi} \to \widehat{T}(H) = {\rm gr}\, \widehat{\K \pi}
\]
such that ${\rm gr}\, \theta = {\rm id}$.
Any group-like expansion induces a filtered $\K$-linear isomorphism $\theta\colon \widehat{\mathfrak{g}} \to {\rm gr}\, \widehat{\mathfrak{g}}$.
This follows from the fact that there is a natural identification
\[
\mathfrak{g} = \K (\pi/{\rm conj}) \cong |\K\pi|:= \K \pi/[\K \pi, \K \pi],
\]
where $\pi/{\rm conj}$ is the set of conjugacy classes in $\pi$ and $[\K \pi, \K \pi]$ is the $\K$-linear span of elements of the form $ab-ba$ with $a,b\in \K \pi$.
Our goal is to characterize group-like expansions which induce Lie bialgebra isomorphisms $\theta\colon \widehat{\mathfrak{g}} \to {\rm gr}\, \widehat{\mathfrak{g}}$.

As was shown by Kawazumi-Kuno \cite{KK14}\cite{KK16} and Massuyeau-Turaev \cite{MT13} \cite{MTpre} independently, if $\theta$ satisfies the boundary condition
\[
\theta(\zeta) = e^{\omega},
\]
where $\zeta \in \pi$ is the boundary loop of $\Sigma$ and $\omega \in \wedge^2 H \subset \widehat{T}(H)$ is the 2-tensor corresponding to the intersection pairing on $\Sigma$, then the induced map $\theta\colon \widehat{\mathfrak{g}} \to {\rm gr}\, \widehat{\mathfrak{g}}$ is a Lie algebra isomorphism.
Group-like expansions satisfying $\theta(\zeta) = e^{\omega}$ are called symplectic expansions (in the case where the boundary of $\Sigma$ is connected).
In this paper, conversely we prove the following theorem:
\begin{thm}[For the general case, see Theorem \ref{thm:main1}]
\label{thm:intromain}
Assume that the boundary of $\Sigma$ is connected.
If a group-like expansion $\theta$ induces a Lie algebra isomorphism $\theta\colon \widehat{\mathfrak{g}} \to {\rm gr}\, \widehat{\mathfrak{g}}$, then $\theta$ is conjugate to a symplectic expansion, i.e., there exists a group-like element $g$ such that
$$
\theta(\zeta)=g e^{\omega} g^{-1}.
$$
\end{thm}

In our previous work \cite{genus0, highergenus}, the formality question for the Lie bialgebra $\mathfrak{g}$ has been studied in connection with the Kashiwara-Vergne problem from Lie theory \cite{KV78, AT12} and its generalization to surfaces of positive genus.
In this approach, one fixes generators of the group $\pi$ and decomposes any group-like expansion as follows:
\[
\theta_F = F^{-1} \circ \theta_{\exp}.
\]
Here, $\theta_{\exp}$ is the group-like expansion determined by the choice of generators of $\pi$ and $F$ is a complete Hopf algebra automorphism of $\widehat{T}(H)$; in other words, $F\in \aut(\widehat{L})$ where $\widehat{L} \subset \widehat{T}(H)$ is the completed free Lie algebra generated by $H$.
In this way, the properties of $\theta$ are encoded in the properties of $F$.
In the formulation of the generalized Kashiwara-Vergne problem in \cite{highergenus}, there are two equations (KV I) and (KV II) for the automorphism $F$.
The equation (KV I) is equivalent to $\theta_F(\zeta) = e^\omega$.
The second equation (KV II) depends on the choice of framing on $\Sigma$ and is related to the formality of the Turaev cobracket.
Recall the following result from \cite{highergenus}:

\begin{thm}[\cite{highergenus}, Theorem 5.12]
\label{thm:introKV}
Suppose that $F \in \aut(\widehat{L})$ satisfies {\rm (KV I)}.
Then, $\theta_F = F^{-1}\circ \theta_{\exp}$ induces a Lie bialgebra isomorphism $\theta_F \colon \widehat{\mathfrak{g}} \to {\rm gr}\, \widehat{\mathfrak{g}}$ if and only if $F$ satisfies {\rm (KV II)}.
\end{thm}

Since the generalized (higher genus) Kashiwara-Vergne problem admits solutions (with the exception of certain framings
on genus one surfaces) \cite[\S 6]{highergenus}, the existence of the Goldman-Turaev formality isomorphism $\widehat{\mathfrak{g}} \to {\rm gr}\, \widehat{\mathfrak{g}}$ has been settled.

In this paper, we introduce a modification of equations (KV I) and (KV II), which we denote by (KV I') and (KV II').
For an explicit form of these two equations, see Theorem \ref{thm:mainGTformal}.
As an application of Theorem \ref{thm:intromain}, we prove the following result:

\begin{thm}[$=$Theorem \ref{thm:mainGTformal}]
\label{thm:introKVmod}
Let $\theta_F=F^{-1}\circ \theta_{\exp}$ be a group-like expansion.
Then, $\theta_F$ induces a Lie bialgebra isomorphism $\theta_F\colon \widehat{\mathfrak{g}} \to {\rm gr}\, \widehat{\mathfrak{g}}$ if and only if $F$ satisfies equations {\rm (KV I')} and {\rm (KV II')}.
\end{thm}

This result is an improvement on Theorem \ref{thm:introKV},  and it gives a complete algebraic characterization of group-like expansions which induce Lie bialgebra isomorphisms $\theta\colon \widehat{\mathfrak{g}} \to {\rm gr}\, \widehat{\mathfrak{g}}$.

The results described above are based on a novel characterization of conjugacy classes in the free Lie algebra. In more detail, let $H$ be a symplectic vector space, $\widehat{L}$ the (degree completed) free Lie algebra generated by $H$, 
$\omega \in \wedge^2 H \subset \widehat{L}$ the element representing the symplectic form, and 
$$
|\widehat{T}(H)| = \widehat{T}(H)/[\widehat{T}(H), \widehat{T}(H)]
$$
the degree completed space of cyclic words with alphabet defined by $H$. We denote the natural projection $\widehat{T}(H) \to |\widehat{T}(H)|$ by $x \mapsto |x|$. We say that an element $a \in \widehat{L}$ is conjugate to $b \in \widehat{L}$ if there is a group-like element $g \in \exp(\widehat{L})$ such that $a=gbg^{-1}$.
The following result plays a key role in the paper:
\begin{thm}[$=$ Theorem \ref{thm:sympl}]
\label{thm:conjugation_intro}
The element $a \in \widehat{L}$ is conjugate to $\omega$ if and only if 
$|\exp(a)| = |\exp(\omega)|$.
\end{thm}

The paper is organized as follows.
In Section \ref{sec:results}, we recall some material from \cite{highergenus} such as the definition of group-like expansions for a surface whose boundary may not be connected, and give a statement of the main result in full generality.
We also prove Theorem \ref{thm:introKVmod} by using Theorem \ref{thm:intromain}.
Section \ref{sec:second} is devoted to conjugation theorems for elements of free Lie algebras. In particular, we prove Theorem \ref{thm:conjugation_intro} modulo some technical statement (Proposition \ref{prop:aomega}). In Section \ref{sec:main_applications}, we give a proof of Theorem \ref{thm:intromain} and discuss applications of this result to non-commutative Poisson geometry. Finally, in Section \ref{sec:Proposition_3_10} we prove Proposition \ref{prop:aomega}.

\vskip 0.2cm

{\bf Acknowledgements.} We are grateful to the Simons Center for Geometry in Physics at the Stony Brook University where part of this work was conducted. Research of AA was supported in part by the project MODFLAT of the European Research Council (ERC), by the grants number 178794 and 178828 and by the NCCR SwissMAP of the Swiss National Science Foundation (SNSF).
Research of NK was supported in part by the grants JSPS KAKENHI 15H03617, 26287006 and 18K03283.
Research of YK was supported in part by the grant JSPS KAKENHI 26800044 and 18K03308.
Research of FN was supported by the grant of Early Postdoc Mobility grant 175033 of the Swiss National Science Foundation.

\tableofcontents

\section{Setup and statement of results}
\label{sec:results}

\subsection{Group-like expansions}

Let $\Sigma=\Sigma_{g,n+1}$ be a compact connected oriented surface of genus $g$ with $n+1$ boundary components, where $g$ and $n$ are non-negative integers.
Label the boundary components of $\Sigma$ by integers $0,1,\ldots,n$, and choose a basepoint $*$ on the $0$th boundary component.
Then the fundamental group $\pi=\pi_1(\Sigma,*)$ has a set of free generators $\alpha_i,\beta_i,\gamma_j$, $i=1,\ldots,g$, $j=1,\ldots,n$, such that $\gamma_j$ is freely homotopic to the $j$th boundary component with positive orientation and
\[
\prod_{i=1}^g \alpha_i \beta_i {\alpha_i}^{-1} {\beta_i}^{-1} \prod_{j=1}^n \gamma_j = \gamma_0,
\]
where $\gamma_0$ is the based loop around the $0$th boundary component with negative orientation.

The first homology group $H=H_1(\Sigma,\K)$ of the surface $\Sigma$ has a 2-step decreasing filtration defined by $H^{(1)}=H$ and
\[
H^{(2)}=\{ x\in H \mid \text{$\langle x,y \rangle = 0$ for all $y\in H$} \},
\]
where $\langle \cdot, \cdot \rangle \colon H\times H \to \K$ is the intersection pairing.
Let
\[
{\rm gr}\, H:= H/H^{(2)} \oplus H^{(2)}
\]
be the associated graded vector space.
The homology classes of $\alpha_i$ and $\beta_i$ give rise to a basis of $H/H^{(2)}$, we denote the corresponding basis elements by $x_i$ and $y_i$. The homology classes of $\gamma_j$, denoted by $z_j$, give rise to a basis of $H^{(2)}$.

Let $A=\widehat{T}({\rm gr}\, H)$ be the completed tensor algebra over ${\rm gr}\, H$.
In other words, $A$ is the completed free associative algebra generated by variables $x_i$, $y_i$, $z_j$.
We assign weights to the generators as follows:
\begin{equation} \label{eq:weights}
{\rm wt}(x_i)={\rm wt}(y_i)=1, \hskip 0.3cm {\rm wt}(z_j)=2.
\end{equation}
Then, the algebra $A$ becomes graded and thus filtered.
Besides, $A$ naturally carries the structure of a complete Hopf algebra.
We denote by $\widehat{L}$ the set of primitive elements in $A$.
It is identified with the completed free Lie algebra generated by ${\rm gr}\, H$.

As was shown in \cite[\S 3.1]{highergenus}, there is a unique multiplicative filtration $\{ \K \pi(m) \}_{m\ge 0}$ of two-sided ideals of the group algebra $\K \pi$ such that $\K \pi(0) = \K \pi$, $\alpha_i-1, \beta_i-1 \in \K \pi(1)$, and $\gamma_j-1 \in \K \pi(2)$.
Furthermore, the (completion of the) associated graded of this filtration is canonically isomorphic to $A$ \cite[Proposition 3.12]{highergenus}.

Let $\widehat{\K \pi} = \varprojlim_m \K\pi/ \K \pi(m)$ be the completion of $\K \pi$ with respect to the filtration described above.
We have ${\rm gr}\, \widehat{\K \pi} = {\rm gr}\, \K \pi =A$.

\begin{dfn}[\cite{highergenus}, Definition 3.19]
\label{dfn:group-like}
A group-like expansion of $\pi$ is an isomorphism $\theta \colon \widehat{\K \pi} \to A$ of complete filtered Hopf algebras such that the associated graded map is the identity: ${\rm gr}\, \theta = {\rm id}$.
\end{dfn}

For example, the map $\theta_{\exp}$ defined by the following values on generators
\[
\theta_{\exp}(\alpha_i) = e^{x_i}, \quad
\theta_{\exp}(\beta_i) = e^{y_i}, \quad
\theta_{\exp}(\gamma_j) = e^{z_j}
\]
is a group-like expansion.
Any group-like expansion $\theta$ can be written as
\[
\theta_F = F^{-1} \circ \theta_{\exp}
\]
for some $F\in \aut(\widehat{L})$ with ${\rm gr} \, F = {\rm id}$.

\subsection{Goldman-Turaev Lie bialgebra and its graded version}

For a (topological) associative $\K$-algebra $\mathfrak{A}$, we denote
\[
|\mathfrak{A}| : =\mathfrak{A}/[\mathfrak{A},\mathfrak{A}],
\]
where $[\mathfrak{A},\mathfrak{A}]$ is the (closure of the) $\K$-span of elements of the form $ab-ba$ with $a,b\in \mathfrak{A}$.
If $\mathfrak{A}$ is filtered, then $|\mathfrak{A}|$ is naturally filtered.
Let $|\cdot |\colon \mathfrak{A} \to |\mathfrak{A}|$ be the natural projection.

The space $\mathfrak{g}=\mathfrak{g}(\Sigma):=|\K \pi|$ is canonically isomorphic to the $\K$-span of  homotopy classes of free loops in $\Sigma$.
As was shown by Goldman \cite{Go86}, the space $\mathfrak{g}$ has a Lie bracket $[\cdot,\cdot]$ defined in terms of intersections of free loops.
By using self-intersections of free loops, Turaev \cite{Tu91} introduced a Lie cobracket on the space $\mathfrak{g}/\K {\bf 1}$, where ${\bf 1}$ denotes the class of a constant loop.
By fixing a framing $f$ on $\Sigma$ (that is, a choice of trivialization of the tangent bundle of $\Sigma$), one can lift it to a Lie cobracket on the space $\mathfrak{g}$, which we denote by $\delta^f$.
The triple $(\mathfrak{g},[\cdot,\cdot],\delta^f)$ becomes a Lie bialgebra \cite[\S 2]{highergenus}.

The Goldman bracket $[\cdot,\cdot]$ and the framed Turaev cobracket $\delta^f$ extend naturally to the Lie bialgebra structure on the completion $\widehat{\mathfrak{g}} = |\widehat{\K \pi}|$.
Moreover, they induce a Lie bialgebra struture on the associated graded space ${\rm gr}\, \widehat{\mathfrak{g}}$, which we denote by $({\rm gr}\, \widehat{\mathfrak{g}}, [\cdot, \cdot]_{\rm gr}, \delta^f_{\rm gr})$.
One can also view the space ${\rm gr}\, \widehat{\mathfrak{g}} \cong |A|$ as the space spanned by cyclic words in $x_i$, $y_i$, $z_j$.
The Lie bracket $[\cdot,\cdot]_{\rm gr}$ and Lie cobracket $\delta^f_{\rm gr}$ coincide with the necklace Lie bialgebra structure associated to the quiver with $g$ circles and $n$ edges emanating from a distinguished vertex, where the Lie bracket was introduced by Bocklandt-Le Bruyn \cite{BB02} and Ginzburg \cite{Ginz01} and the Lie cobracket by Schedler \cite{Schedler}.
Any group-like expansion $\theta$ induces an isomorphism
\[
\theta \colon \widehat{\mathfrak{g}} = |\widehat{\K \pi}| \overset{\cong}{\to} |A| = {\rm gr}\, \widehat{\mathfrak{g}}
\]
of complete filtered $\K$-vector spaces.

\subsection{The main result: Goldman formality revisited}

\begin{dfn}[\cite{highergenus}, Definition 3.21]
A group-like expansion $\theta$ is called tangential if for any $j=1,\ldots,n$, there is a group-like element $g_j \in A$ such that $\theta(\gamma_j) = g_j e^{z_j} {g_j}^{-1}$.
Furthermore, a tangential group-like expansion $\theta$ is called special if $\theta(\gamma_0) = e^{\omega}$, where $\omega = \sum_i [x_i,y_i] + \sum_j z_j$.
\end{dfn}

Note that the elements $\omega_0 = \sum_i [x_i,y_i]$ and $\sum_j z_j$ (once we choose the 0th boundary component of $\Sigma$) are independent of the choice of generators $\alpha_i, \beta_i, \gamma_j$, and, hence, so is the element $\omega$.

\begin{rem}
For the defining conditions for special expansions, see also \cite[\S 7.2]{KK16}.
For $n=0$, the boundary condition $\theta(\gamma_0) = e^{\omega_0}$ was first turned into a definition by Massuyeau \cite{Mas12}.
In this case, special expansions are called symplectic expansions.
Special expansions exist for any $g$ and $n$.
For examples for $g=0$ or $n=0$, see \cite{HM00, AET10, Kaw06, Kun12}.
For the general case, there is a gluing argument which proves existence of special expansions, see \cite[\S 3.5]{highergenus} for details.
\end{rem}

Let us recall some results on special expansions and the Goldman bracket.

\begin{thm}[Kawazumi-Kuno \cite{KK14, KK16}, Massuyeau-Turaev \cite{MT13,MTpre}]
\label{thm:KKMT}
Every special expansion $\theta$ induces a Lie algebra isomorphism between the completed Goldman Lie algebra $(\widehat{\mathfrak{g}}, [\cdot, \cdot])$ and its associated graded, $({\rm gr}\, \widehat{\mathfrak{g}},[\cdot,\cdot]_{\rm gr})$.
\end{thm}

The main result of this paper is  the converse of Theorem \ref{thm:KKMT} (up to conjugacy):

\begin{thm}
\label{thm:main1}
Let $\theta\colon \widehat{\K \pi} \to A$ be a group-like expansion and assume that $\theta$ induces a Lie algebra isomorphism between $(\widehat{\mathfrak{g}}, [\cdot, \cdot])$ and $({\rm gr}\, \widehat{\mathfrak{g}},[\cdot,\cdot]_{\rm gr})$.
Then, $\theta$ is conjugate to a special expansion.
Namely, $\theta$ is tangential and there exists a group-like element $g_0 \in A$ such that $\theta(\gamma_0) = g_0 e^{\omega} {g_0}^{-1}$.
\end{thm}

\begin{rem}
Denote by $\iota_{j}\colon \Z=\pi_1(S^1) \to \pi/\text{conj}$ the maps $\iota_j\colon t^n \mapsto |\gamma_j^n|$ induced by inclusions of the  boundary components into $\Sigma$. These maps are independent of the concrete choice of generators $\gamma_j \in \pi$ and they are compatible with filtrations if one assigns a filtration degree $2$ to the generator $t \in \mathbb{Z}$. Note that the complete Hopf algebra $\widehat{\mathbb{K}\mathbb{Z}}$ admits a unique group-like expansion $\theta_{\exp}\colon \widehat{\mathbb{K}\mathbb{Z}} \to \mathbb{K}[[\tau]]$ given by 
$\theta_{\exp}(t)=\exp(\tau)$, where $\tau$ is the primitive generator of degree $2$. The map $\iota_j$ induces a map of associated graded
${\rm gr} \, \iota_j\colon \mathbb{K}[[\tau]] \to |A|$.
With this notation, Theorem \ref{thm:main1} is equivalent to the following statement:

\begin{quote}
Let $\theta\colon \widehat{\mathbb{K}\pi} \to A$ be a group-like expansion. Then, $\theta$ induces a Lie algebra isomorphism between $(\widehat{\mathfrak{g}}, [\cdot, \cdot])$ and $({\rm gr}\, \widehat{\mathfrak{g}},[\cdot,\cdot]_{\rm gr})$ if and only if it preserves all the boundary components in the sense that for $j = 0, \ldots, n$ the following diagram commutes:
\begin{displaymath}
    \xymatrix{
        \widehat{\mathfrak{g}} \ar[r]^\theta  & {\rm gr} \, \widehat{\mathfrak{g}} \\
        \widehat{\mathbb{K} \Z} \ar[u]_{\iota_j} \ar[r]^{\theta_{\exp}}       & \mathbb{K}[[\tau]] \ar[u]_{{\rm gr}\, \iota_j}}
\end{displaymath}
\end{quote}
\end{rem}

Let ${\rm Aut}(\widehat{L})$ be the group of filtration preserving automorphisms of $\widehat{L}$.
We say that $F\in {\rm Aut}(\widehat{L})$ is tangential if for any $j=1,\ldots,n$ there is a group-like element $f_j \in A$ such that $F(z_j) = {f_j}^{-1} z_j f_j$.
Note that any $F\in {\rm Aut}(\widehat{L})$ extends to a filtration preserving automorphism of $A=U(\widehat{L})$, and thus induces a filtration preserving automorphism of $|A|={\rm gr}\, \widehat{\mathfrak{g}}$.
Also, since $F$ is filtration preserving, ${\rm gr}\, F$ is defined as an automorphism of ${\rm gr}\, \widehat{L}=\widehat{L}$.

It turns out that Theorem \ref{thm:main1} can be restated as a property of the automorphism group of the Lie algebra 
$({\rm gr}\, \widehat{\mathfrak{g}}, [\cdot, \cdot]_{\rm gr})$.

\begin{thm}
\label{thm:main2}
Let $F\in {\rm Aut}(\widehat{L})$ such that ${\rm gr}\, F = {\rm id}$ and
assume that it induces an automorphism of the Lie algebra $({\rm gr}\, \widehat{\mathfrak{g}},[\cdot,\cdot]_{\rm gr})$.
Then, $F$ is tangential and there exists a group-like element $f_0 \in A$ such that $F(\omega) = {f_0}^{-1} \omega f_0$.
\end{thm}

\begin{proof}
Let $\theta_0: \widehat{\K \pi} \to A$ be a special expansion. Then, by Theorem \ref{thm:KKMT}, it induces a Lie algebra isomorphism between $(\widehat{\mathfrak{g}}, [\cdot, \cdot])$ and $({\rm gr}\, \widehat{\mathfrak{g}}, [\cdot, \cdot]_{\rm gr})$. If $F\in {\rm Aut}(\widehat{L})$ satisfies assumptions of Theorem \ref{thm:main2}, then the map $F\circ \theta_0$ a group-like expansion and it also induces a Lie algebra isomorphism between $(\widehat{\mathfrak{g}},[\cdot,\cdot])$ and $({\rm gr}\, \widehat{\mathfrak{g}},[\cdot,\cdot]_{\rm gr})$.
By Theorem \ref{thm:main1}, we conclude that $F\circ \theta_0$ is conjugate to a special expansion. Therefore, $F=(F\circ \theta_0) \circ \theta_0^{-1}$ maps $z_j$'s and $\omega$ to their conjugates, as required.
\end{proof}

\subsection{Goldman-Turaev formality and Kashiwara-Vergne equations}

In this section, we combine Theorem \ref{thm:main1} with our previous result \cite{genus0, highergenus} and derive a necessary and sufficient condition for group-like expansions which induce Goldman-Turaev formality.

In \cite[\S 5.4]{highergenus},
for any compact connected oriented surface of genus $g$ with $n+1$ boundary components and a choice of framing $f$ on it, we introduced the Kashiwara-Vergne problem $\KV^{(g,n+1)}_f$, which asks to find a tangential automorphism of $\widehat{L}$ satisfying two equations ($\KV$ I) and ($\KV$ II).
The original Kashiwara-Vergne problem \cite{KV78, AT12} corresponds to the case of $(g,n+1) = (0,3)$.

More concretely, the first equation is of the form
\[
F(\omega) =
\log \left( \prod_{i=1}^g (e^{x_i} e^{y_i} e^{-x_i} e^{-y_i}) \prod_{j=1}^n e^{z_j} \right) =: \xi.
\tag{KV I}
\]
Since $\theta_{\exp}(\gamma_0) = e^{\xi}$, $F$ satisfies ($\KV$ I) if and only if $\theta_F=F^{-1}\circ \theta_{\exp}$ is a special expansion.

To write down the second equation, we need more material from \cite{highergenus}.
Let
\[
\tder^+ = (\tder^{(g,n+1)})^+=\{ (u,u_1,\ldots,u_n) \in \der^+(\widehat{L}) \times \widehat{L}^{\oplus n} \mid
u(z_j) = [z_j,u_j] \}
\]
be the Lie algebra of tangential derivations on $\widehat{L}$ and
let
\[
\taut = \taut^{(g,n+1)}=\{ (F,f_1,\ldots,f_n) \in \aut^+(\widehat{L}) \times \widehat{L}^{\oplus n} \mid F(z_j) = e^{-f_j}z_j e^{f_j} \}
\]
be the group of tangential automorphisms on $\widehat{L}$.
The divergence cocycle ${\rm div}\colon \tder^+ \to |A|$ is defined as follows:
\[
{\rm div}(u) := \left|\,  \sum_{i=1}^g(\pa_{x_i}u(x_i) + \pa_{y_i}u(y_i) )
+ \sum_{j=1}^n \pa_{z_j}u(z_j)\,  \right|
\in |A|.
\]
Here, for any $a\in A$ without constant term, we denote
\begin{equation}
\label{eq:pader}
a = \sum_{i=1}^g ((\pa_{x_i}a)\, x_i + (\pa_{y_i}a)\, y_i ) + \sum_{j=1}^n (\pa_{z_j}a)\, z_j.
\end{equation}
Now, choose a framing $f$ on $\Sigma$.
For any immersed closed curve $\gamma$ in $\Sigma$, one can define its rotation number $\rot_f(\gamma)\in \mathbb{Z}$ with respect to $f$.
Put
\[
p:=\sum_{i=1}^g (\rot_f(|\beta_i|)\, x_i - \rot_f(|\alpha_i|)\, y_i) \in H/H^{(2)} \subset {\rm gr}\, H \subset A
\]
and define the cocycle $c_f\colon \tder^+ \to |A|$ by
\[
c_f(u)= \sum_j \rot_f(|\gamma_j|) |u_j|.
\]
The cocycles ${\rm div}$ and $c_f$ integrate to group $1$-cocycles
\[
j \colon \taut \to |A| \quad \text{and} \quad C_f \colon \taut \to |A|.
\]
More explicitly, if $F=\exp(u) \in \taut$ with $u\in \tder^+$ then $j(F) = \frac{e^u-1}{u}\cdot \div(u)$.
Set
\[
j_f:= j-C_f.
\]
Finally, introduce the element ${\bf r}\in |A|$ by
\[
{\bf r} = \sum_{i=1}^g \left| \log \big( \frac{e^{x_i}-1}{x_i} \big) + \log \big (\frac{e^{y_i}-1}{y_i} \big) \right|.
\]
With the notation as above, the second Kashiwara-Vergne equation is of the form
\[
j_f(F) = {\bf r}+|p|+\sum_{j=1}^n |h_j(z_j)| - |h(\xi)| \quad \text{for some $h_j,h\in \K[[s]]$.}
\tag{KV II}
\]

We recall the following result:

\begin{thm}[\cite{highergenus}]
\label{thm:KVII}
Let $f$ be a framing on $\Sigma$ and
$F \in \taut$ be a solution of the equation {\rm (KV I)}.
Then, $\theta_F= F^{-1} \circ \theta_{\rm exp}$ is a Lie bialgebra isomorphism from $(\widehat{\mathfrak{g}},[\cdot,\cdot],\delta^f)$ to $({\rm gr}\, \widehat{\mathfrak{g}},[\cdot,\cdot]_{\rm gr},\delta^f_{\rm gr})$ if and only if $F$ satisfies {\rm (KV II)}.
\end{thm}

Combining Theorem \ref{thm:main1} and Theorem \ref{thm:KVII}, we obtain the following result:

\begin{thm}
\label{thm:mainGTformal}
Let $\theta$ be a group-like expansion.
Then, $\theta$ induces a Lie bialgebra isomorphism between
$(\widehat{\mathfrak{g}},[\cdot,\cdot],\delta^f)$ and $({\rm gr}\, \widehat{\mathfrak{g}},[\cdot,\cdot]_{\rm gr},\delta^f_{\rm gr})$ if and only if there exists a tangential automorphism $F\in \taut$ such that $\theta = \theta_F = F^{-1} \circ \theta_{\exp}$,
\[
F(\omega) = e^{-\ell_0} \xi e^{\ell_0}
\quad \text{for some $\ell_0\in \widehat{L}$},
\tag{KV I'}
\]
and
\[
j_f(F) + \rot_f(|\gamma_0|)\, |\ell_0|
= {\bf r} + |p| + \sum_{j=1}^n |h_j(z_j)| - |h(\xi)|
\quad \text{for some $h_j,h\in \K[[s]]$}.
\tag{KV II'}
\]
\end{thm}

\begin{rem}
Since $|\cdot |$ vanishes on commutators, the expression $|\ell_0|$ is of degree $1$.
We can ignore the terms in $ \ell_0$ proportional to $z_j$'s because they can be absorbed in the linear terms of functions $h_j$.
\end{rem}

\begin{proof}
First, we compute the cocycle $j_f$ on inner automorphisms.
For $\ell \in \widehat{L}$ let $u_\ell \in \tder^+$ be the inner derivation with generator $\ell$: $u_{\ell}(a)=[a,\ell]$ for any $a\in \widehat{L}$.
The element $F_\ell :=\exp(u_\ell) \in \taut$ is an inner automorphism given by conjugation by $e^{\ell}$: $F_\ell(a) = e^{-\ell} a e^\ell$.

We compute
\begin{align*}
    {\rm div}(u_\ell) &=
    \left| \sum_{i=1}^g ( \pa_{x_i}[x_i,\ell]+\pa_{y_i}[y_i,\ell]) + \sum_{j=1}^n  (\pa_{z_j}[z_j, \ell]) \right| \\
    &= \left| \sum_{i=1}^g (x_i (\pa_{x_i}\ell)-\ell + y_i (\pa_{y_i}\ell)-\ell) + \sum_{j=1}^n (z_j (\pa_{z_j}\ell) - \ell) \right| \\
    &= (1-2g-n) |\ell|.
\end{align*}
In the third line we used the cyclic invariance $|x_i (\pa_{x_i}\ell)|= |(\pa_{x_i}\ell) x_i|$ and formula \eqref{eq:pader} for $\ell$.
Also, we have $c_f(u_\ell)=(\sum_j {\rm rot}(|\gamma_j|) )|\ell|$.
Since $u_\ell$ acts trivially on the space $|A|$, integration yields
\begin{equation}
\label{eq:jqFh}
j_f(F_\ell) = ({\rm div} -c_f)(u_\ell)
=(1-2g-n-\sum_{j=1}^n {\rm rot}(|\gamma_j|))\, |\ell| = -\rot_f(|\gamma_0|)\, |\ell|.
\end{equation}
In the last equality we have used the Poincar\'e-Hopf theorem
\[
\sum_{j=1}^n \rot_f(|\gamma_j|) - \rot_f(|\gamma_0|) = \chi(\Sigma) = 1-2g-n.
\]

Now let $\theta$ be a group-like expansion and assume that it induces a Lie bialgebra isomorphism $\theta\colon \widehat{\mathfrak{g}} \to {\rm gr}\, \widehat{\mathfrak{g}}$.
By Theorem \ref{thm:main1}, there exists an element $F=(F,f_1,\ldots,f_n) \in \taut$ and a group-like element $g_0\in A$ such that $\theta=\theta_F = F^{-1} \circ \theta_{\exp}$ and $\theta(\gamma_0) = g_0 e^{\omega} {g_0}^{-1}$.
By setting $\ell_0:= \log F(g_0) \in \widehat{L}$ we obtain $F(\omega) = e^{-\ell_0} \xi e^{\ell_0}$, which implies (KV I').

By construction, the automorphism $F':=F_{-\ell_0}\circ F$ satisfies (KV I).
Since the action of $F_{-\ell_0}$ on $|A|$ is trivial, $\theta_{F'} = \theta$ on $\widehat{\mathfrak{g}}$.
By Theorem \ref{thm:KVII}, $F'$ satisfies (KV II).
This implies that $F$ satisfies (KV II'), since
\[
j_f(F') = j_f(F_{-\ell_0})+F_{-\ell_0}\cdot j_f(F)
=\rot_f(|\gamma_0|)\, |\ell_0| +j_f(F)
\]
by \eqref{eq:jqFh} and the fact that $F_{-\ell_0}$ acts trivially on $|A|$.
This completes the proof of ``only if'' part.
The other direction can be proved by the same method, so we omit it.
\end{proof}

\section{Free Lie algebras and cyclic words}
\label{sec:second}

In this section, we will prove several statements about conjugacy classes in free Lie algebras and their characterization in terms of cyclic words. These statements are the main technical content of the paper.

\subsection{PBW type decompositions}

Let $\gfg$ be a Lie algebra over a field $\K$ of characteristic zero.
Later we will take $\gfg$ to be the free Lie algebra 
over a finite dimensional $\K$-vector space.
By the PBW theorem, we have a natural decomposition
\begin{equation}
U(\gfg) = \bigoplus^\infty_{m=0} \sym^m\gfg,
\label{eq:PBW}
\end{equation}
where we denote by $\sym^m\gfg$ the $m$th symmetric power of the vector space
$\gfg$. The isomorphism \eqref{eq:PBW} is given by the maps 
$\sym^m\gfg \to U(\gfg)$,
$$
x_1x_2\cdots x_m \mapsto \dfrac{1}{m!}\sum_{\sigma\in\mathfrak{S}_m}
x_{\sigma(1)}x_{\sigma(2)}\cdots x_{\sigma(m)},
$$
for $m \geq 0$. Here $x_1x_2\cdots x_m$ stands for the symmetric tensor product 
of $x_1, x_2, \dots,  x_m \in \gfg$ (see, {\em e.g.} \cite{Bou71} Ch.\ 1, \S2, no.\ 7).
It should be remarked that the vector space $\sym^m\gfg$ is spanned by the set $\{x^m; \,\, x \in \gfg\}$. 
Now we consider the Lie algebra abelianization of the associative algebra
$U(\gfg)$,
$$
\vert U(\gfg)\vert = U(\gfg)/[U(\gfg), U(\gfg)].
$$
We denote the quotient map by 
$\vert \cdot\vert\colon U(\gfg) \to \vert U(\gfg)\vert$, 
$u \mapsto \vert u\vert$. 
The decomposition \eqref{eq:PBW} descends to abelianizations:
\begin{thm}\label{thm:assert} We have the direct sum decomposition 
\begin{equation}
\vert U(\gfg)\vert = \bigoplus^\infty_{m=0} \vert \sym^m\gfg\vert.
\label{eq:PBWtype}
\end{equation}
\end{thm}
\begin{proof}
Since $\gfg$ generates $U(\gfg)$, we have $[U(\gfg), U(\gfg)] = [\gfg, U(\gfg)]$.
Recall that the decomposition \eqref{eq:PBW} is a $\gfg$-module decomposition.
Hence, we have 
$$
[\gfg, U(\gfg)] = \sum^\infty_{m=0} [\gfg, \sym^m\gfg]
\subset \sum^\infty_{m=0} ([\gfg, U(\gfg)] \cap \sym^m\gfg).
$$
This means that the subspace $[\gfg, U(\gfg)] = [U(\gfg), U(\gfg)]$ is homogeneous with respect to the decomposition \eqref{eq:PBW}, 
and this implies the direct sum decomposition in the theorem.
\end{proof}
\par
\bigskip
Let $V$ be a finite dimensional $\K$-vector space, 
$T=T(V) =\bigoplus^{\infty}_{m=0} V^{\otimes m}$
the tensor algebra over $V$,
and $L=L(V)$ the free Lie algebra over $V$.
If we denote by $\Delta\colon T\to T \otimes T$ the (standard) coproduct of the Hopf algebra structure on $T$, then $L$ is identified with the set of primitive elements, i.e., $L=\{ a\in T;\, \Delta a = a\otimes 1 + 1 \otimes a\}$, and we have $T=U(L)$.
For our purpose we need completions of $T$ and $L$;
we denote by
$\widehat{T} = \widehat{T}(V) = \prod^\infty_{m=0}V^{\otimes m}$
the completed tensor algebra over $V$ and
by $\widehat{L} = \widehat{L}(V)$ the completed free Lie algebra over $V$.

The Lie algebra $L$ admits a grading with finite dimensional graded components given by tensor powers of $V$:
$L = \bigoplus_{q=0}^\infty (L \cap V^{\otimes q})$.
This implies that decompositions \eqref{eq:PBW} and \eqref{eq:PBWtype} extend to $\widehat{L}$ (with direct sums replaced by direct products):
$$
\widehat{T}=U(\widehat{L}) = \prod_{m=0}^\infty \sym^m \widehat{L}, \hskip 0.3cm
\vert \widehat{T} \vert = \prod_{m=0}^\infty \vert \sym^m \widehat{L} \vert .
$$
This observation has the following interesting corollaries:

\begin{thm}\label{thm:dec}
Let $u$ and $v \in \widehat{L}$ such that
$
\vert \exp(u) \vert = \vert \exp(v)\vert \in \vert \widehat{T} \vert .
$
Then,
$$
\vert u^m\vert = \vert v^m\vert \in \vert \widehat{T} \vert
$$
for all $m \geq 0$. 
\end{thm}
\begin{proof} 
We have,
$$
\vert \exp(u) \vert = \sum_{m=0}^\infty \frac{1}{m!} \vert u^m\vert = \sum_{m=0}^\infty \frac{1}{m!} \vert v^m\vert =
\vert \exp(v) \vert.
$$
By decomposition \eqref{eq:PBWtype} for the Lie algebra $\widehat{L}$, this implies that the series in the middle are equal term by term:
$$
\vert u^m \vert = \vert v^m \vert,
$$
as required.
\end{proof}

Similarly, one can prove the following statement:
\begin{thm}\label{thm:dec2}
Let $u$ and $v_0, \dots, v_n \in \widehat{L}$ satisfy 
$
\vert \exp(u)\vert \in \left\vert\sum^n_{j=0}\K[[v_j]]\right\vert.
$
Then, we have
$
\vert u^m\vert \in \sum^n_{j=0}\K\vert {v_j}^m\vert
$
for all $m \geq 0$.
\end{thm}

\begin{proof} 
Observe that 
$$
\left\vert \sum^n_{j=0}\K[[v_j]] \right\vert \cap \left\vert
\sym^m \widehat{L} \right\vert = \sum^n_{j=0}\K\vert {v_j}^m\vert .
$$
The component of $\vert \exp(u) \vert$ in $\left\vert
\sym^m \widehat{L} \right\vert$ is $\vert u^m \vert/m!$. Hence, we conclude 
$$
\vert u^m \vert \in \sum^n_{j=0}\K\vert {v_j}^m\vert,
$$
as required.
\end{proof}

\subsection{Conjugacy theorems}

This subsection is devoted to two conjugacy theorems.
As in the preceding sections, for a $\K$-vector space $V$, 
we denote by $\widehat{T}=\widehat{T}(V)$ the completed tensor algebra over $V$, and by
$\widehat{L}=\widehat{L}(V)$ the completed free Lie algebra over $V$.

We are now ready to formulate the main technical results of the paper:

\begin{thm}
\label{thm:non-conn}
Let $V$ be a finite dimensional $\K$-vector space.
Suppose that 
an element $z\in \widehat{L}$ has non-trivial linear term, 
$z \in \widehat{L} \setminus\prod_{l>1}( \widehat{L} \cap V^{\otimes l})$, 
and that another element $a \in \widehat{L}$ satisfies
$\vert\exp(a)\vert = \vert\exp(z)\vert \in \vert \widehat{T}\vert$.
Then we have 
$a = gzg^{-1}$ for some group-like element $g \in \exp(\widehat{L})$.
\end{thm}

\begin{thm}
\label{thm:sympl}
Let $V$ be a finite dimensional $\K$-symplectic vector space, 
whose symplectic form we denote by $\omegaz \in \wedge^2 V$.
Suppose that an element $a \in \widehat{L}$ satisfies
$\vert\exp(a)\vert = \vert\exp(\omegaz)\vert \in \vert \widehat{T}\vert$.
Then we have 
$a = g\omegaz g^{-1}$ for some group-like element $g \in \exp(\widehat{L})$.
\end{thm}

To prove these theorems we need some preliminary lemmas.
Let $V$ be a finite dimensional $\K$-vector space.
\par
\begin{lem}\label{lem:n2}
Let $z$ be an element of the sets $V\setminus\{0\}$ or 
$\wedge^2V\setminus\{0\}$. Then, we have 
$$
\{u \in \widehat{T}\otimes \widehat{T}; \,\, [\Delta z, u] = 0\} = \K[[z]]\otimes\K[[z]].
$$
\end{lem}
\begin{proof}
It suffices to show that the LHS is included in the RHS. 
As was proved in Proposition 5.6, \cite{highergenus}, we have
\begin{equation}
\{a \in \widehat{T}; \,\, [z, a] \in \K[[z]]\} = \K[[z]].
\label{eq:norm}
\end{equation}
In fact, the element $z$ is {\it reduced} in the sense of \S5.2 in \cite{highergenus}. 
\par
Choose a basis $\{x_1,\dots, x_n\}$ of $V$
(with no relation to $z$).
Let $u$ be an element of the LHS. 
We may assume that under the grading defined by powers of $V$ it is homogeneous of some degree $m \geq 0$. 
One can write uniquely
$$
u = u_0\otimes 1 + \sum^m_{k=1}u_{i_1\dots i_k}\otimes x_{i_1}\dots x_{i_k}, \quad
u_0 \in V^{\otimes m},\,\,u_{i_1\dots i_k} \in V^{\otimes (m-k)}.
$$
Then, since $z$ is primitive, we have 
\begin{equation}
0 = [\Delta z, u] = [z, u_0]\otimes 1 + 
\sum^m_{k=1}([z, u_{i_1\dots i_k}]\otimes x_{i_1}\dots x_{i_k}
+ u_{i_1\dots i_k}\otimes [z, x_{i_1}\dots x_{i_k}]).
\label{eq:eqn2}
\end{equation}

\par
We claim that $u_{i_1\dots i_k} \in \K[[z]]$ for all $k$ and $i_1\dots i_k$.
First consider the case $z \in V\setminus\{0\}$. 
Then, equation \eqref{eq:eqn2} is equivalent to the following family of equations
\begin{eqnarray*}
&&0 = [z, u_0] \otimes 1, \\
&&0 = \sum_{i_1}[z,u_{i_1}] \otimes x_{i_1} \quad\text{and} \\
&&0 = \sum_{i_1,\dots,i_k}[z, u_{i_1\dots i_k}]\otimes x_{i_1}\dots x_{i_k}
+ \sum_{i_1,\dots,i_{k-1}}u_{i_1\dots i_{k-1}}\otimes [z, x_{i_1}\dots x_{i_{k-1}}]
\end{eqnarray*}
for $k \geq 2$. By \eqref{eq:norm}, we have $u_0 \in \K[[z]]$ and $u_{i_1} 
\in \K[[z]]$ for any $i_1$. Suppose $k \geq 2$, and assume $u_{i_1\dots i_{k-1}} 
\in \K[[z]]$ for all $i_1\dots i_{k-1}$. Then, the third equation above implies
$[z, u_{i_1\dots i_k}] \in \K[[z]]$. Again, by \eqref{eq:norm}, 
we obtain $u_{i_1\dots i_k} \in \K[[z]]$. This completes the induction.\par
In the case $z \in \wedge^2V\setminus\{0\}$,
equation \eqref{eq:eqn2} is equivalent to the following family of equations:
\begin{eqnarray*}
&&0 = [z, u_0] \otimes 1, \\
&&0 = \sum_{i_1}[z,u_{i_1}] \otimes x_{i_1}, \\
&&0 = \sum_{i_1,i_2} [z, u_{i_1,i_2}] \otimes x_{i_1}x_{i_2} 
\quad\text{and} \\
&&0 = \sum_{i_1,\dots,i_k}[z, u_{i_1\dots i_k}]\otimes x_{i_1}\dots x_{i_k}
+ \sum_{i_1,\dots,i_{k-2}}u_{i_1\dots i_{k-2}}\otimes [z, x_{i_1}\dots x_{i_{k-2}}]
\end{eqnarray*}
for $k \geq 3$. The same argument as above applies 
to give $u_{i_1\dots i_k} \in \K[[z]]$. 
\par
In both cases, we have $u_{i_1\dots i_k} = \lambda_{i_1\dots i_k}z^{m-k}$ 
for some $\lambda_{i_1\dots i_k} \in \K$. Hence, if we write 
$v_k = \sum_{i_1,\dots,i_k}\lambda_{i_1\dots i_k}x_{i_1}\dots x_{i_k}
\in V^{\otimes k}$, we have
$u = \sum^m_{k=0}z^{m-k}\otimes v_k$, which implies
$$
0 = [\Delta z, u] = \sum^m_{k=0}z^{m-k}\otimes [z, v_k].
$$
Again, by \eqref{eq:norm}, $v_k \in \K[[z]]$.
This completes the proof of the lemma.
\end{proof}

\begin{lem}\label{lem:n3}
Let $z$ be an element of the sets $V\setminus\{0\}$ or 
$\wedge^2V\setminus\{0\}$. Then, we have 
$$
\{a \in \widehat{T}; \,\, [z, a]  \in \widehat{L}\} = \widehat{L} + \K[[z]].
$$
\end{lem}
\begin{proof}
It suffices to show that the LHS is included in the RHS.\par
By the PBW decomposition \eqref{eq:PBW}, we can consider 
the projection
$
\varpi_1: T = U(L) = \bigoplus^\infty_{m=0}\sym^m L
\to \sym^1 L = L
$,
and a linear endomorphism
$E: T = \bigoplus^\infty_{m=0}\sym^m L \to T$ 
defined by 
$$
E\vert_{\sym^m L} = 
\begin{cases}
\dfrac{1}{m}1_{\sym^m L}, & \text{if $m \geq 2$},\\
0, & \text{if $m \leq 1$}.\\
\end{cases}
$$
If we denote the multiplication by $\mu: T\otimes T \to T$ and 
use the symbol
$$
\ssym(\ell_1, \dots, \ell_m)
:= \sum_{\sigma\in \mathfrak{S}_m}
\ell_{\sigma(1)}\cdots \ell_{\sigma(m)}
$$
for $\ell_i \in L$, $1 \leq i \leq m$, then 
one deduces that 
$\mu(1_T\otimes\varpi_1)(\Delta a 
- a\otimes 1
- 1\otimes a)\vert_{a=\ssym(\ell_1, \dots, \ell_m)}$
equals $m(\ssym(\ell_1, \dots, \ell_m))$ if $m\geq 2$, 
and $0$ if $m \leq 1$.
Hence, we have 
\begin{equation}
E\mu(1_T\otimes\varpi_1)(\Delta a - a\otimes 1 - 1\otimes a)
= \begin{cases}
a & \text{for $a \in \bigoplus^\infty_{m=2}\sym^m L$, and}\\
0 & \text{for $a \in \sym^0 L \oplus\sym^1 L$}.\\
\end{cases}
\label{eq:E}
\end{equation}
\par
Assume that $a \in \widehat{T}$ satisfies $[z, a] \in \widehat{L}$. 
We may assume $a$ is homogeneous, in particular, 
that $a$ is an element of $T = U(L)$.
Moreover we may assume 
$a \in \bigoplus^\infty_{m=2} \sym^m L$. 
If we write 
$u = \Delta a - a\otimes 1 - 1\otimes a$, 
then 
$$
0 = \Delta([z, a]) - [z, a]\otimes 1 - 1\otimes [z, a]
= [\Delta z, \Delta a] - [z\otimes 1+1\otimes z, a\otimes 1+1\otimes a]
= [\Delta z, u].
$$
Hence, Lemma \ref{lem:n2} implies that $u = \sum_{i,j=0}^\infty \lambda_{ij}z^i\otimes
z^j$ for some $\lambda_{ij} \in \K$, and so
\begin{equation}
\Delta a - a\otimes 1 - 1\otimes a 
= \sum_{i,j=0}^\infty \lambda_{ij}z^i\otimes z^j.
\label{eq:Da}
\end{equation}
Applying \eqref{eq:E} to \eqref{eq:Da}, we obtain 
$
a = E\mu(1_T\otimes\varpi_1)(\sum_{i,j=0}^\infty \lambda_{ij}z^i\otimes z^j)
= E\mu(\sum_{i=0}^\infty \lambda_{i1}z^i\otimes z)
= E(\sum_{i=0}^\infty \lambda_{i1}z^{i+1}) \in \K[[z]].
$
This completes the proof.
\end{proof}
\par
Now we can begin the proof of Theorem \ref{thm:non-conn}.
One of the keys to the proof is the following.
\begin{prop}[Proposition A.2 in \cite{genus0}]
\label{prop:A2}
Let $x \in V\setminus \{0\}$ and $a \in \widehat{T}$. 
If $|ax^l|=0$ for all $l \geq 1$, then $a\in [x,\widehat{T}]$.
\end{prop}
\begin{proof}[Proof of Theorem \ref{thm:non-conn}]
It suffices to prove the theorem in the case $z \in V\setminus\{0\}$.
Indeed,  write $z = z_1 + z'$ with $z_1 \in V\setminus\{0\}$ and
$z' \in \widehat{L}\cap \prod_{l>1}V^{\otimes l}$. 
Then, by the universal mapping property of $\widehat{L}$, 
there is a continuous Lie algebra endomorphism $\varphi$ of $\widehat{L}$ 
such that $\varphi(z_1) = z$, $\varphi(\prod_{l>p}V^{\otimes l}) 
\subset \prod_{l>p}V^{\otimes l}$ for any $p \geq 0$, and 
the associated graded of $\varphi$ with respect to the filtration 
$\{\prod_{l>p}V^{\otimes l}\}_{p \geq 0}$ is the identity.
From these properties one can deduce that $\varphi$ is a topological 
automorphism of $\widehat{L}$. Thus the theorem for $z_1$ 
implies that for $z$.\par
For the rest of the proof, we suppose $z \in V\setminus\{0\}$. 
We denote the Baker-Campbell-Hausdorff series by $\ast: \widehat{L}\times \widehat{L} \to \widehat{L}$,
$(u, v) \mapsto u\ast v = \operatorname{bch}(u, v)$. Namely we have 
$\exp(u\ast v) = \exp(u)\exp(v) \in \widehat{T}$. \par
From Theorem \ref{thm:dec} follows 
\begin{equation}
|a^m| = |z^m|
\label{eq:am}
\end{equation}
for any $m \geq 1$.
By induction on $k \geq 1$, we will prove that there exist elements $u_k \in \widehat{L}\cap V^{\otimes k}$ such that 
$$
\aligned
& \exp(\ad(u_k))\exp(\ad(u_{k-1}))\cdots\exp(\ad(u_1))(a) \\
(= &
\exp(u_k\ast u_{k-1}\ast \cdots \ast u_1)(a)\exp(u_k\ast u_{k-1}\ast \cdots \ast u_1)^{-1} )\\
\equiv & \,\,z 
\pmod{\prod^\infty_{l>k+1}V^{\otimes l}}.
\endaligned
$$
Since $|a| = |z|$, we have $a \equiv z + b_2 \pmod{\prod^\infty_{l>2}V^{\otimes l}}$ for some $b_2 \in \widehat{L}\cap V^{\otimes 2}$. The component of degree $(m+1)$ in equation \eqref{eq:am} reads $m|b_2 z^{m-1}| = 0$ for any $m \geq 2$. Hence, 
by Proposition \ref{prop:A2}, we have $b_2 = [z, u_1]$ for some $u_1 \in V
= \widehat{L}\cap V^{\otimes 1}$ and
$$
\exp(\ad(u_1))(a) \equiv 
z+ b_2 + [u_1, z] = z \pmod{\prod^\infty_{l>2}V^{\otimes l}}.
$$ 
Suppose $k \geq 2$. By the inductive assumption we have
$$
\exp(\ad(u_{k-1}))\cdots\exp(\ad(u_1))(a) \equiv z + b_{k+1}
\pmod{\prod^\infty_{l>k+1}V^{\otimes l}}
$$
for some $b_{k+1} \in \widehat{L}\cap V^{\otimes (k+1)}$ and $u_1, \dots, u_{k-1}
\in \widehat{L}$. The degree $(m+k)$ part of equation \eqref{eq:am} reads $m|b_{k+1} z^{m-1}| = 0$ for any $m \geq 2$.
Hence, by Proposition \ref{prop:A2}, we have 
$b_{k+1} = [z, u'_k]$ for some $u'_k \in V^{\otimes k}$. 
Applying Lemma \ref{lem:n3} to equation $[z, u'_k] = b_{k+1} \in \widehat{L}$, 
we obtain $u'_k = u_k + \lambda_k z^k$, where $u_k \in \widehat{L}\cap V^{\otimes k}$ and $\lambda_k \in \K$. Therefore, $b_{k+1} = [z, u_k]$ and 
$$
\exp(\ad(u_k))\exp(\ad(u_{k-1}))\cdots\exp(\ad(u_1))(a) \equiv 
z+ b_{k+1} + [u_k, z] = z \pmod{\prod^\infty_{l>k+1}V^{\otimes l}},
$$ 
as required. \par
The sequence $\{v_k= u_k\ast u_{k-1}\ast \cdots \ast u_1\}_{k=1}^\infty$
converges to an element $v_\infty \in \widehat{L}$ by degree counting. 
Taking $g = \exp(-v_\infty) \in \exp(\widehat{L})$, we obtain 
$g^{-1}ag = z \in \widehat{L}$. This completes the proof.
\end{proof}
\par
The proof of Theorem \ref{thm:sympl} is quite similar to that of Theorem 
\ref{thm:non-conn}, so we omit it except for a symplectic analogue 
of Proposition \ref{prop:A2}.
\begin{prop}
\label{prop:A2sympl}
Let $V$ be a finite dimensional $\K$-symplectic space with symplectic form  $\omegaz \in \wedge^2V$.  
If an element $a \in \widehat{T}$ satisfies $|a\omegaz^l|=0$ for all $l \geq 1$, then there is an element $b \in \widehat{T}$ such that $a=[\omegaz,b]$.
\end{prop}

In order to prove this statement, we may assume that $a$ is homogeneous. Thus, it is sufficient to prove the following proposition.

\begin{prop}
\label{prop:aomega}
Let $a \in V^{\otimes m}$ for some $m\ge 0$.
Assume that for some $p\ge 1$ we have $|a\omegaz^l| = 0$ for all $l\ge p$.
Then, there is an element $b \in V^{\otimes m-2}$ such that $a=[\omegaz,b]$.
\end{prop}

The proof of this proposition is postponed to Section \ref{sec:Proposition_3_10}.

\section{Proof of Theorem \ref{thm:main1} and applications}
\label{sec:main_applications}

In this section, we prove Theorem \ref{thm:main1} and explain some applications to non-commutative Poisson geometry.

\subsection{Proof of Theorem \ref{thm:main1}}

We consider the situation of Section \ref{sec:results} and use the notation introduced there.
We apply results of the preceding section to $V={\rm gr}\, H$ and $A=\widehat{T}(V)=\widehat{T}({\rm gr}\, H)$.
Note that if $n>0$, the expression $\omega = \omegaz + \sum_{j=1}^n z_j$ has a non-trivial linear term,
and if $n=0$, then $\omega = \omegaz = \sum_i [x_i,y_i]$ is a symplectic form on $V$.

Recall that as a vector space the associated graded of the Goldman Lie algebra ${\rm gr} \, \widehat{\mathfrak{g}}$ is isomorphic to $|A| = |\widehat{T}({\rm gr}\, H)|$.
The following theorem gives a description of its center:

\begin{thm}[Theorem 5.4 in \cite{highergenus}]
\label{thm:center}
$$
Z({\rm gr} \, \widehat{\mathfrak{g}}, [\cdot, \cdot]_\gr) = |\, \K[[\omega]]\, | 
\oplus\bigoplus^n_{j=1} |\, \K[[z_j]]_{\geq1}\, |.
$$
\end{thm}
This result has been proved in \cite{CBEG07} by using  Poisson geometry of quiver varieties. An alternative elementary proof is given 
in \cite[\S5.4]{highergenus}.

Recall that for $1\le j\le n$ we denote by $\gamma_j$ the loop along the $j$th boundary component (with positive orientation), and that $\gamma_0$ is the loop along the $0$th boundary component (with negative orientation).

\begin{prop}
\label{prop:411}
Let $\theta: \widehat{\mathbb{K}\pi} \to A$ be a group-like expansion and assume 
$\vert\theta(\gamma_j)\vert \in Z({\rm gr} \, \widehat{\mathfrak{g}}, [\cdot,\cdot]_{\gr})$.
Then, we have 
$$
\vert \theta(\gamma_j)\vert = 
\begin{cases}
\vert \exp(z_j)\vert, & \text{for $j \geq 1$},\\
\vert \exp(\omega)\vert, &\text{for $j=0$}.
\end{cases}
$$
\end{prop}
\begin{proof}
Recall the grading on $A$ in which $x_i$ and $y_i$ have degree $1$ and $z_j$ has degree 2.
Under this grading, $z_j$'s and the expression $\omega = \omega_0 + \sum_{j=1}^n z_j$ are homogeneous and have degree 2.
By Theorem \ref{thm:center}, we have 
$$
|\theta(\gamma_j)| \in |\, \K[[\omega]]\, | 
\oplus\bigoplus^n_{j=1} |\, \K[[z_j]]_{\geq1}\, |.
$$
By Theorem \ref{thm:dec2}, for any $m\ge 0$ we have
$$
|(\log \theta(\gamma_j))^m| \in \mathbb{K}|\omega^m| \oplus \bigoplus_{j=1}^n \mathbb{K}|{z_j}^m|.
$$
Note that all the terms on the right hand side have degree exactly $2m$. Furthermore, note that
$$
\log\theta(\gamma_j) \equiv
\begin{cases}
z_j, & \text{if $j \geq 1$},\\
\omega, &\text{if $j=0$},
\end{cases}
\pmod{{\rm terms \, of \, degree} \, \geq 3 }.
$$
Therefore,
$$
|(\log\theta(\gamma_j))^m| \equiv
\begin{cases}
|{z_j}^m|, & \text{if $j \geq 1$},\\
|\omega^m|, &\text{if $j=0$},
\end{cases}
\pmod{{\rm terms \, of \, degree} \, \geq 2m+1 }.
$$
But $|(\log\theta(\gamma_j))^m|$ contain no terms of degree higher than $2m$. Hence, the equalities above are verified without error terms of higher degree which proves the proposition.
\end{proof}

\begin{proof}[Proof of Theorem \ref{thm:main1}]
Let $\theta \colon \widehat{\K\pi} \to A$ be a group-like expansion which induces a Lie algebra
isomorphism $(\widehat{\mathfrak{g}},  [\cdot, \cdot])
\overset\cong\to ({\rm gr} \, \widehat{\mathfrak{g}}, [\cdot,\cdot]_{\gr})$. 
Since for each boundary loop $\gamma_j$ the expression $|{\gamma_j}|$, $0 \leq j \leq n$,
is in the center for the Goldman bracket, 
we have $|\theta(\gamma_j)| \in Z({\rm gr} \, \widehat{\mathfrak{g}}, [\cdot,\cdot]_{\gr})$.
Hence, by Proposition \ref{prop:411}, $|\theta(\gamma_j)|=|\exp(\log\theta(\gamma_j))|$ equals 
$|\exp(z_j)|$ if $j \geq 1$, and $|\exp(\omega)|$
if $j = 0$. \par
Suppose $j \geq 1$ or $n \geq 1$. Then, by Theorem \ref{thm:non-conn}, 
we have some group-like element $g_j$ such that $\log\theta(\gamma_j)$
equals $g_jz_j{g_j}^{-1}$ for $j \geq 1$, and $g_0 \omega {g_0}^{-1}$ for $j=0$.
In the case of $n=0$, one has $\omega=\omegaz$ and Theorem \ref{thm:sympl} implies that
$\log\theta(\gamma_0) = {g_0}\omega{g_0}^{-1}$ for some group-like element $g_0$. 
Therefore, the expansion $\theta$ is conjugate to a special expansion which proves the theorem.
\end{proof}

\subsection{An application to non-commutative Poisson geometry}

Recall the context of non-commutative differential calculus. Let $A=\mathbb{K}\langle\langle u_1, \dots, u_s\rangle\rangle$ be a free associative algebra with $s$ even generators, and 
$$
D_A^\bullet = \mathbb{K}\langle\langle u_1, \dots, u_s, \partial_1, \dots, \partial_s \rangle\rangle
$$
be the free associative algebra with $s$ even generators $u_1, \dots, u_s$ and $s$ odd generators $\partial_1, \dots, \partial_s$. The algebra $D_A^\bullet$ carries a double bracket in the sense of van den Bergh defined by formula
$$
\{ \partial_i, u_j \} = \delta_{i,j} 1 \otimes 1, \hskip 0.3cm
\{ \partial_i, \partial_j\} = \{ u_i, u_j\} =0.
$$
The space of cyclic words
$$
|D_A^\bullet| = D^\bullet_A/[D_A^\bullet, D_A^\bullet] = \bigoplus_{k=0}^\infty |D_A^k|
$$
carries the induced graded Lie bracket (the non-commutative analogue of the Schouten bracket on polyvector fields) \cite{vdB}. Note that $|D_A^0|=|A|$, $D_A^1={\rm Der}(A, A\otimes A)$ is the space of double derivations of $A$, $|D_A^1|={\rm Der}(A, A)$ is the Lie algebra of derivations of $A$, and $|D_A^2|$ is the space of double brackets on $A$. A double bracket $\Pi \in |D_A^2|$ is Poisson if and only if the non-commutative Schouten bracket vanishes: $[\Pi, \Pi]=0$. A Poisson double bracket $\Pi$ induces a differential $d_\Pi=[\Pi, \cdot]$ on $|D^\bullet_A|$.

There is a natural map 
$$
\partial: |D_A^k| \to {\rm Hom}_\mathbb{K}(|A|^{\otimes k}, |A|)
$$
defined by differentiating $k$ elements of $|A|$ by $k$ double derivations $\partial_i$ contained in an element of $|D_A^k|$. Note that the right hand side also carries a Schouten bracket and the map $\partial$ is a Lie homomorphism.

Let $E$ be the double derivation defined by formula $E(a)=a\otimes 1 - 1 \otimes a$ for every $a\in A$. Define the graded quotient  space $\mathcal{D}_A^\bullet$ as follows
\begin{align*}
    \mathcal{D}_A^\bullet = \operatorname{coker}( D_A^{\bullet-1} \overset{\alpha \mapsto |\alpha E|}{\longrightarrow} |D_A^\bullet|).
\end{align*}
That is, $\mathcal{D}_A^k=|D_A^k|/|D_A^{k-1} E|$. 

\begin{prop} \label{prop:E_acts_zero}
The map $\partial$ vanishes on $|D^\bullet_A E|$.
\end{prop}

\begin{proof}
Since $E(a)=a \otimes 1 - 1 \otimes a$, we have $E(|a|) = a-a=0$. Hence, $|D^{k-1}_A E|$ acts by zero on $|A|^{\otimes k}$.
\end{proof}
Proposition \ref{prop:E_acts_zero} implies that the map 
$\partial$ descends to a map (which we denote by the same letter)
\begin{align}
    \partial: \mathcal{D}_A^\bullet &\longrightarrow {\rm Hom}_\K( |A|^\bullet, |A|).
\end{align}

\begin{prop} \label{prop:DE_ideal}
The subspace $|D^\bullet_A E| \subset |D_A^\bullet|$ is a Lie ideal under the Schouten bracket.
\end{prop}

\begin{proof}
We compute,
$$
[ |a|, |Eb| ] = |\{ |a|, E\} b| + |E \{ |a|, b\}| = |E \{ |a|, b\} | \in 
|D_A^\bullet E|,
$$
where we have used $\{ |a|, E\} = - E(|a|)=0$. This shows that $|D^\bullet_A E|$ is indeed a Lie ideal in $|D^\bullet_A|$.
\end{proof}
As a simple example, consider $|D_A^0 E|$. This space is spanned by elements of the form $|aE|$ which induce inner derivations on $A$: $\{ |aE|, b\} = ab- ba$. The Lie algebra $|D^1_A|$ is isomorphic to the Lie algebra derivations of $A$ and, hence, $\mathcal{D}^1_A$ is isomorphic to the Lie algebra of outer derivations of $A$.

 Proposition \ref{prop:DE_ideal} implies that the Schouten bracket descends to $\mathcal{D}^\bullet_A$ and makes the map $\partial$ above into a Lie homomorphism. If $\Pi \in |D^2_A|$ is a Poisson double bracket, the differential $d=[\Pi, \cdot]$ also descends to $\mathcal{D}_A^\bullet$ and defines a non-commutative Poisson cohomology theory.
 
 We now choose $A=\widehat{T}({\rm gr} H)$, where ${\rm gr} H$ is the associated graded of the first homology of the surface $\Sigma$ with generators $x_i, y_i$ of degree 1 and generators $z_j$ of degree 2. We consider the double bracket (see Section 5.1 in \cite{highergenus})
 $$
 \Pi = \sum_{i=1}^g |\partial_{x_i} \partial_{y_i}| + \sum_{j=1}^n |z_j \partial_{z_j} \partial_{z_j}|.
 $$
 The first term on the right hand side is the symplectic double bracket induced by the intersection pairing on the first homology, and the second term is the Kirillov-Kostant-Souriau (KKS) linear bracket corresponding to boundary components. It induces the associated graded of the Goldman Lie bracket. 
 
 \begin{rem}
 In this paper, we assume double brackets to be skew-symmetric. This explains the difference in the form of the bivector $\Pi$ above and the bivector $\Pi_{\rm gr}$ in \cite{highergenus} (where no skew-symmetry assumption was made).
 \end{rem}

The main result of this section is the following theorem:

\begin{thm}   \label{thm:Poisson_cohomology}
\begin{align*}
H^0(\mathcal{D}_A^\bullet, d= [\Pi, \cdot]) &= Z(|A|), \\
H^1(\mathcal{D}_A^\bullet, d= [\Pi, \cdot]) &= \bigoplus_i \K |\partial_{z_i}|.
\end{align*}
\end{thm}

\begin{rem}
Let ${\rm Rep}(A, N) = {\rm Hom}(A, \operatorname{End}(\K^N))$ be the space of $N$-dimensional representations of $A$. In \cite{vdB}, van den Bergh constructs a Lie homomorphism $|D_A^\bullet| \to T_\text{poly}({\rm Rep}(A, N))^{GL_N}$. Under this correspondence, elements of the form $|\alpha E|$ are in the image of the map $ (\mathfrak{gl}_N \otimes T_\text{poly}^{\bullet-1} \to T_\text{poly}^\bullet )^{GL_N}$ given by the $\mathfrak{gl}_N$ action by conjugation. Hence,  $\mathcal{D}_A^\bullet$ maps to polyvector fields on the quotient space ${\rm Rep}(A, N)/GL_N$ and the non-commutative Poisson cohomology $H(\mathcal{D}_A^\bullet, d= [\Pi, \cdot])$ maps to the Poisson cohomology of ${\rm Rep}(A, N)/GL_N$.
\end{rem}
 
\begin{proof}[Proof of Theorem \ref{thm:Poisson_cohomology}]
In degree zero, the cohomology of $d$ are elements $|a| \in |A| =\mathcal{D}_A^0$ such that $\{ |a|, \cdot \}$ is an inner derivation of $A$. By Theorem A.1 in \cite{genus0}, this is equivalent to $|a|$ being in the center of the associated graded of the Goldman bracket, as required.

In degree one, let $u \in |D_A^1|$ and $[u] \in \mathcal{D}_A^1 = \der(A,A)/ \operatorname{inn}(A,A)$ such that $d([u]) = 0$. That is, the class of $[\Pi, u]$
in $\mathcal{D}_A^2$ vanishes and therefore $\partial([\Pi, u])=[\partial(\Pi), \partial(u) ]=0$. Then, $u$  induces the derivation $\partial(u)$ of the Lie bracket $ \partial(\Pi): |A| \otimes |A| \to |A|$. 

Assume that $u \in \der^{\geq -1}(A)$. Then, by Lemma \ref{lem:fully_tangential}, $u$ is tangential and $u(\omega)=[\omega, l] $ for some $l$. Hence, $u$ is special up to an inner derivation. As shown in \cite[\S 5.1]{highergenus}, any special derivation is Hamiltonian. That is, there exists $|a| \in |A|$ such that $u = \{|a|, \cdot\}$ which implies $[u]=d |a|$. 

If $u$ is of degree $(-2)$, it is of the form $\sum_j \lambda_j |\partial_{z_j}|$. That is, $u(x_i)=u(y_i)=0, u(z_j)=\lambda_j$. Note that 
$$
[ |z \partial_z \partial_z|, |\partial_z| ] = |\partial_z \partial_z|=0,
$$
where we have used the fact that $\partial_z$ is odd. Note also that the image of $d: \mathcal{D}_A^0 \to \mathcal{D}_A^1$ is a graded vector space with degree bounded from below by $(-1)$. Hence, all the derivations of degree $(-2)$ define nontrivial cocycles, as required.
\end{proof}

Assume $(n,g) \neq (1,0)$.
Recall that a derivation $\phi \in \der(A)$ is called tangential if 
$\phi(z_i)=[a_i, z_i]$ for all $i=1, \dots, n$ for some $a_i \in A$. 
Set $z_0 = -\omega = -\omega_0 - \sum_{j=1}^n z_j$. We say that $\phi \in \der(A)$ is fully tangential if it is tangential and in addition 
$\phi(z_0)=[a_0, z_0]$ for some $a_0 \in A$.

\begin{lem}  \label{lem:fully_tangential}
Assume that  $\phi \in \der^{\geq -1}(A)$ induces a derivation of the graded Goldman bracket on $|A|$. Then, $\phi$ is a fully tangential derivation.
\end{lem}
\begin{proof}
Let $\phi \in \der(A)$ be such that it induces a derivation of the graded Goldman bracket. Then, it preserves the center of the graded Goldman Lie algebra. Recall that the center is spanned by elements $|{z_j}^k|$ for $j = 0, \ldots, n$ and $k\ge 0$. In particular, for all $l$ and $N$ we have
\begin{align*}
    |\phi(z_l^N)| = N |z_l^{N-1} \phi(z_l)| = \sum_{j=0}^n |f_j(z_j)|,
\end{align*}
for some $f_j \in \K [[s]]$ of degree at least $N-1$. Setting $z_l = 0$ in the equation above, we obtain for $N\geq 2$
\begin{align*}
 \sum_{j \neq l} |f_j(z_j)| =0.
\end{align*}
Since $|{z_j}^k|$ are linearly independent for $k \geq 2$, we obtain that for $N \geq 3$,
\begin{align*}
    N |z_l^{N-1} \phi(z_l)| = |f_l(z_l)|.
\end{align*}
Using Propositions \ref{prop:A2} and \ref{prop:A2sympl} we conclude that $\phi(z_l)$ is of the form
\begin{align*}
    \phi(z_l) = [a_l, z_l] + g_l(z_l)
\end{align*}
for some $a_l \in A$ and $g_l \in \K[[s]]$.
Using the relation $|\sum_{j=0}^n z_j|=0$ we obtain that
\begin{align*}
    0 = \left| \phi(\sum_{j=0}^n z_j) \right| = \left| \sum_{j=0}^n  ([a_j, z_j] + g_j(z_j)) \right|= \left| \sum_{j=0}^n g_j(z_j) \right|.
\end{align*}
This equality implies that functions $g_j(z_j)$ are at most linear.
Since derivations of weight $(-2)$ were excluded by assumptions, we have $g_j(z_j)= 2\lambda z_j$ for some $\lambda \in \mathbb{K}$. 
Then, define a derivation $\psi$ of $A$ by
\begin{align*}
    \psi(x_i) &=  \lambda x_i, \quad i=1,\ldots,g \\
    \psi(y_i) &= \lambda y_i, \quad i=1,\ldots,g \\
    \psi(z_j) &= 2 \lambda z_j, \quad j=1,\ldots,n.
\end{align*}
The difference $\phi - \psi$ is now fully tangential, hence preserves the graded Goldman bracket, and thus $\psi$ preserves the graded Goldman bracket. This implies that the graded Goldman bracket is of weight  zero. Since it is actually of weight $(-2)$, this is a contradiction unless the Goldman bracket vanishes identically (which is only the case if $(n,g)=(1,0)$).
\end{proof}

\section{Proof of Proposition \ref{prop:aomega}}
\label{sec:Proposition_3_10}

In the proof of Proposition \ref{prop:aomega},
it will be convenient to identify 
$|\widehat{T}|$ with a vector subspace of cyclically invariant elements of $\widehat{T}$ through
the embedding $|\widehat{T}| \hookrightarrow \widehat{T}$
defined by 
$$
|x_1x_2\cdots x_l| \mapsto \sum^{l-1}_{k=0}\nu^k(x_1x_2\cdots x_l).
$$
Here $x_k \in V$ for $1 \leq k \leq l$ and $\nu$ is the cyclic permutation:
$$
\nu: x_1x_2\cdots x_l \mapsto x_2\cdots x_l x_1.
$$
Recall that for $l = 0,1$ we have $|V^{\otimes l}| = V^{\otimes l}$. 

Let ${\rm dim} V=2g$ be the dimension of the symplectic vector space $V$ and $C: V^{\otimes 2} \to \mathbb{K}$ the non-degenerate  pairing defined by the symplectic form $\omega_0$.
Denote $Q={\rm Ker}(C)$ and let $\pi: V^{\otimes 2} \to Q \subset V^{\otimes 2}$
be the projection corresponding to the direct sum decomposition $V^{\otimes 2} = Q \oplus \mathbb{K} \omegaz$.

In what follows, we use the following simple facts: 
$C(\omegaz) = 2g$, $\pi(\omegaz)=0$, and
\begin{equation}
\label{eq:CllC}
(C^{\otimes l} \otimes 1_V) (x\omegaz^l) = (1_V \otimes C^{\otimes l}) (\omegaz^l x) = (-1)^l x
\end{equation}
for any $x\in V$ and $l\ge 1$.

First, we prove Proposition \ref{prop:aomega} for $m=0,1,2,3$.

{\it The case $m=0$.}
Since $|\omegaz^l | = l (\omegaz^l + \nu (\omegaz^l))$,
\[
C^{\otimes l} |\omegaz^l| = l ( (2g)^l + (-1)^l 2g ).
\]
If $l\ge 2$, the right hand side of this equation is nonzero.
Hence, $|a\omegaz^l|=0$ implies $a=0$, as required.

{\it The case $m=1$.}
If $\deg a = 1$, then
$|a\omegaz^l| = \sum_{j=0}^l \omegaz^j a \omegaz^{l-j} + \sum_{j=0}^{l-1} \nu(\omegaz^j a\omegaz^{l-j})$.
Since
\[
1_V \otimes C^{\otimes l} :
\omegaz^j a \omegaz^{l-j} \mapsto (-1)^j (2g)^{l-j} a,
\quad
\nu(\omegaz^j a \omegaz^{l-j}) \mapsto (-1)^{l-j} (2g)^j a,
\]
we have
\begin{align*}
(1_V \otimes C^{\otimes l}) |a\omegaz^l| & = \sum_{j=0}^l (-1)^j (2g)^{l-j}a + \sum_{j=0}^{l-1} (-1)^{l-j} (2g)^j a \\
& = \left( (2g)^l + 2 (-1)^l \sum_{j=0}^{l-1} (-2g)^j \right) a.
\end{align*}
If $l\ge 2$, the coefficient of $a$ is not zero.
Hence, $|a\omegaz^l|=0$ (for $l$ sufficiently large) implies $a=0$.

{\it The case $m=2$.}
If $\deg a = 2$, then
$|a\omegaz^l| = \sum_{j=0}^l \omegaz^j a \omegaz^{l-j} + \sum_{j=0}^l \nu (\omegaz^j a \omegaz^{l-j})$.
We have
\[
\pi \otimes C^{\otimes l} :
\begin{cases}
a\omegaz^l \mapsto (2g)^l \pi(a) & \\
\omegaz^j a \omegaz^{l-j} \mapsto 0 & \text{for any $1\le j\le l$} \\
\nu(\omegaz^j a \omegaz^{l-j}) \mapsto (-1)^l \pi(\nu(a)) & \text{for any $0\le j\le l$}
\end{cases}
\]
Here, the second case follows from $\pi(\omegaz)=0$ and
the third case from \eqref{eq:CllC}.
Therefore,
\[
(\pi \otimes C^{\otimes l})|a\omegaz^l| = (2g)^l \pi(a) + (l+1) (-1)^l \pi(\nu(a)).
\]
Now assume that $|a\omegaz^l|=0$ for any $l\ge p$.
Then, the right hand side of the above formula vanishes for any $l\ge p$, and this shows that $\pi(a)=0$ and $\pi(\nu(a))=0$.
The equation $\pi(a)=0$ implies that $a$ is a multiple of $\omegaz$. From the case $m=0$, we deduce that $a=0$.

To consider the case of $m=3$, we need the following lemma.
\begin{lem}
\label{lem:H3}
$V^{\otimes 3} = [\omegaz,V] \oplus (V\otimes Q)$.
\end{lem}

\begin{proof}
Let $a\in [\omegaz, V] \cap (V\otimes Q)$.
Since $a\in V\otimes Q$, $(1_V \otimes C)(a)=0$.
On the other hand, there is an element $b\in V$ such that
$a=[\omegaz,b]=\omegaz b - b\omegaz$, and
$(1_V\otimes C)(a)= -b -(2g)b=-(2g+1)b$.
Therefore, $b=0$ and $a=0$.
Thus $[\omegaz, V] \cap (V\otimes Q) = \{ 0\}$.
By counting dimensions, the assertion follows.
\end{proof}

{\it The case $m=3$.}
In view of Lemma \ref{lem:H3}, it is sufficient to prove the
following assertion:
let $a\in V\otimes Q$ and assume that there exists some $p\ge 1$ such that $|a\omegaz^l| = 0$ for any $l\ge p$. Then $a=0$.

Assume $a \in V\otimes Q$ and introduce the notation $a=a_1a_2a_3$.
Let us apply $1_V \otimes \pi \otimes C^{\otimes l}$ to 
\[
|a\omegaz^l | =
\sum_{j=0}^l \omegaz^j a \omegaz^{l-j} + \sum_{j=0}^l \nu(\omegaz^j a \omegaz^{l-j}) + a_2a_3 \omegaz^l a_1.
\]
Since
\[
1_V \otimes \pi \otimes C^{\otimes l} :
\begin{cases}
a\omegaz^l \mapsto (2g)^l a & \\
\omegaz^j a \omegaz^{l-j} \mapsto 0 & \text{for any $1\le j\le l$} \\
\nu(a\omegaz^l) \mapsto (-1)^l (1_V \otimes \pi)(a_3a_1a_2) & \\
\nu(\omegaz^j a \omegaz^{l-j}) \mapsto 0 & \text{for any $1\le j\le l$} \\
a_2a_3\omegaz^l a_1 \mapsto (-1)^l (1_V\otimes \pi)(a_2a_3a_1)
\end{cases},
\]
we have
\[
0=(1_V \otimes \pi \otimes C^{\otimes l})|a\omegaz^l| =
(2g)^l a + (-1)^l (1_V\otimes \pi) (a_3a_1a_2 + a_2a_3a_1).
\]
Since this equality holds true for any $l\ge p$, we deduce that
$a=0$.

The case $m\ge 4$ can be solved inductively based on the following proposition.

\begin{prop}
\label{prop:QHQ}
Let $m\ge 4$, $a\in Q\otimes V^{\otimes m-4} \otimes Q$, and
$b\in V^{\otimes m-2}$.
Assume that there exists some $p\ge 1$ such that $|a\omegaz^l +b \omegaz^{l+1}| =0$ for all $l \geq p$.
Then, $a=0$.
\end{prop}

{\it The case $m\ge 4$.}
Let $a\in V^{\otimes m}$ and assume that
$|a\omegaz^l| = 0$ for any $l\ge p$.
By the direct sum decomposition $V^{\otimes 2}=Q \oplus \mathbb{K} \omegaz$, we can uniquely write
\[
a=\omegaz b' + b'' \omegaz + c,
\]
where $b'\in V^{\otimes m-2}$, $b'' \in Q\otimes V^{\otimes m-4}$, and $c \in Q\otimes V^{\otimes m-4} \otimes Q$.
For any $l\ge p$,
\[
0=|a\omegaz^l| = |(\omegaz b' + b'' \omegaz + c) \omegaz^l|
= | c \omegaz^l + (b' + b'')\omegaz^{l+1} |.
\]
Applying Proposition \ref{prop:QHQ}, we obtain $c=0$,
$a=\omegaz b' + b'' \omegaz$, and $|(b'+b'')\omegaz^l | =0$ for any $l\ge p+1$.
By the inductive assumption, there exists some $d\in V^{\otimes m-4}$ such that $b'+b''=[\omegaz,d]$.
Then
\[
a=\omegaz b' + ([\omegaz,d]-b')\omegaz =
[\omegaz,b'] + [\omegaz, d] \omegaz = [\omegaz, b'+d\omegaz],
\]
as required.
This completes the proof of Proposition \ref{prop:aomega}.

Finally, let us prove Proposition \ref{prop:QHQ}.
We use the following two lemmas, which can be proved by straightforward computations.

\begin{lem}
\label{lem:Codd}
Let $m$ be an odd integer $\ge 5$ and $l\ge (m+1)/2$.
For any $u_1,u_2,\ldots,u_m \in V$, we have
\begin{align*}
& (\pi \otimes {1_V}^{\otimes m-4} \otimes \pi) ({1_V}^{\otimes m} \otimes C^{\otimes l}) |u_1u_2\cdots u_m \omegaz^l| \\
=& (\pi \otimes {1_V}^{\otimes m-4} \otimes \pi)
( (2g)^l u_1u_2\cdots u_m + (-1)^l \Phi(u)),
\end{align*}
where
\begin{align*}
&\Phi(u) =
\sum_{\substack{1\le k\le m-1 \\ k\, {\rm odd}}}
(-1)^{\frac{k-1}{2}}
\left(
{\rm Cont}(u_{m-k+1}\cdots u_{m-1})\, u_m \omegaz^{\frac{k-1}{2}} u_1 u_2 \cdots u_{m-k} \right. \\
& \hspace{15em} \left. +{\rm Cont}(u_2\cdots u_k)\, u_{k+1}\cdots u_m \omegaz^{\frac{k-1}{2}} u_1 \right).
\end{align*}
Here, ${\rm Cont}(u_2\cdots u_k)=C(u_2,u_3)\cdots C(u_{k-1},u_k) \in \K$.
\end{lem}

\begin{lem}
\label{lem:Ceven}
Let $m$ be an even integer $\ge 4$ and $l\ge m/2$.
For any $u_1,u_2,\ldots,u_m \in V$, we have
\begin{align*}
& (\pi \otimes {1_V}^{\otimes m-4} \otimes \pi) ({1_V}^{\otimes m} \otimes C^{\otimes l}) |u_1u_2\cdots u_m \omegaz^l | \\
=& (\pi \otimes {1_V}^{\otimes m-4} \otimes \pi)
( (2g)^l u_1u_2\cdots u_m + \left( l-\frac{m}{2} \right) (-1)^l
\Phi_1(u)+ (-1)^l \Phi_2(u) ),
\end{align*}
where
\begin{align*}
& \Phi_1(u) = (-1)^{\frac{m}{2}-1}
{\rm Cont}(u_2\cdots u_{m-1})\, u_m \omegaz^{\frac{m}{2}-1} u_1, \\
& \Phi_2(u)
= \sum_{\substack{1\le k\le m-1 \\ k\, {\rm odd}}} (-1)^{\frac{1}{2}(k-1)}
\left( {\rm Cont}(u_{m-k+1}\cdots u_{m-1})\, u_m \omegaz^{\frac{k-1}{2}} u_1u_2\cdots u_{m-k} \right. \\
& \hspace{15em} \left. + {\rm Cont}(u_2\cdots u_k)\, u_{k+1}\cdots u_m \omegaz^{\frac{k-1}{2}} u_1 \right).
\end{align*}
\end{lem}

\begin{proof}[Proof of Proposition \ref{prop:QHQ}]
First assume that $m$ is odd and $\ge 5$.
We apply Lemma \ref{lem:Codd} to $u=a+b\omegaz$.
Since $a\in Q \otimes V^{\otimes m-4} \otimes Q$,
$(\pi \otimes {1_V}^{\otimes m-4} \otimes \pi)(a)=a$.
Since $\pi(\omegaz)=0$,
$(\pi \otimes {1_V}^{\otimes m-4} \otimes \pi)(b\omega_0)=0$.
Hence
\[
0= (2g)^l a + (-1)^l (\pi \otimes {1_V}^{\otimes m-4} \otimes \pi) \Phi(a+b\omegaz)
\]
for any $l\gg 0$.
Therefore, $a=0$.

Next, assume that $m$ is even and $\ge 4$.
We apply Lemma \ref{lem:Ceven} to $u=a+b\omegaz$.
Then, we obtain
\[
0= (2g)^l a + \left( l-\frac{m}{2} \right) (-1)^l \Phi_1(a+b\omegaz) + (-1)^l \Phi_2(a+b\omegaz)
\]
for any $l\ge p$.
We can find sufficiently large integers $l_1,l_2,l_3$ such that
\[
\det \begin{pmatrix}
(2g)^{l_1} & (l_1-\frac{m}{2})(-1)^{l_1} & (-1)^{l_1} \\
(2g)^{l_2} & (l_2-\frac{m}{2})(-1)^{l_2} & (-1)^{l_2} \\
(2g)^{l_3} & (l_3-\frac{m}{2})(-1)^{l_3} & (-1)^{l_3} \\
\end{pmatrix}
\neq 0.
\]
Therefore, we can conclude that $a=0$.
\end{proof}

\end{document}